\documentclass[12pt]{article}

\usepackage[affil-it]{authblk}      \usepackage{latexsym}
         \usepackage[reqno, namelimits, sumlimits]{amsmath}
         \usepackage{amssymb, amsfonts}
         \usepackage{amsthm}
 \newtheorem{theorem}{Theorem}[section]
 
 \newtheorem{definition}[theorem]{Definition}

 \newtheorem{pro}[theorem]{Proposition}

\title{Remark on Wolf's condition for  boundary regularity of Navier-Stokes equations}
\author{ G Seregin
  }
\affil{OxPDE, Mathematical Institute, University of Oxford, Oxford,UK}
\date{ \today}
\begin{document}
\maketitle
\abstract{We prove Wolf's regularity condition up to the boundary for solutions to the Navier-Stokes equations satisfying non-slip boundary condition.}\setcounter{equation}{0}
\section{Introduction}

The note is inspired by an interesting result by J. Wolf, see \cite{Wolf2010}.  It reads the following. Let a pair $u$ and $p$ be a suitable weak solution to the Navier-Stokes system in the parabolic cylinder
$Q(z_0,R)=B(x_0,R)\times ]t_0-R^2,t_0[$, where $B(x_0,R)$ is a ball of radius $R$ centred at point $x_0\in\mathbb R^3$. There exists a positive constant $\varepsilon$ such that if
$$\frac 1{R^2}\int\limits_{Q(z_0,R)}|u(z)|^3dz<\varepsilon$$
then $u\in L_\infty(Q(z_0,R/2))$.

At the first glance, the condition and the result are independent of pressure $p$. But it is wrong impression as one can see from an elementary example in which
$R=1$, $z_0=(0,0)$, 
$$u(x,t)=c(t)\nabla h(x),  \qquad p(x,t)=-c'(t)h(x)+\frac 12 c^2(t)|\nabla h(x)|^2,$$ 
and $h$ is a harmonic function. If there is no restriction on pressure, then the above assumption does not provide regularity. But of course there  is an assumption on the pressure that is hidden in the definition of suitable weak solution. Let me recall it.
\begin{definition}
\label{suitableweakinterior}
A suitable weak solution $u$ and $p$ to the classical Navier-Stokes system in $Q(z_0,R)$ possess the following properties:
%\begin{equation}\label{spacesin}
$$u\in L_{2,\infty}(Q(z_0,R))\cap W^{1,0}_2(Q(z_0,R)),\qquad p\in L_\frac 32(Q(z_0,R));$$
%\end{equation}
%\begin{equation}\label{NSSin}
$$\partial_tu+u\cdot\nabla u-\Delta u=-\nabla p,\qquad{\rm div}\,u=0$$
%\end{equation}
in $Q(z_0,R)$ in the sense of distributions; 
%\begin{equation}\label{boundaryconditions}
%u(x',0,t)=0
%\end{equation}
%for all $|x'-x'_0|<R$ and $t\in ]t_0-R^2,t_0[$, where $x'=(x_1,x_2)$;

for a.a. $t\in ]t_0-R^2,t_0[$,
$$\int\limits_{B(x_0,R)}\varphi^2(x,t)|u(x,t)|^2dx+2\int\limits^t_{t_0-R^2}\int\limits_{B(x_0,R)}\varphi^2|\nabla u|^2dxds\leq$$
%\begin{equation}\label{localenergyinequalityin}
$$\leq\int\limits^t_{t_0-R^2}\int\limits_{B(x_0,R)}|u|^2(\partial_t\varphi^2+\Delta\varphi^2)+u\cdot\nabla \varphi^2(|u|^2+2p)dxds$$
%\end{equation}
for all $\varphi \in C^\infty_0(B(x_0,R)\times ]t_0-R^2,t_0+R^2[)$.
\end{definition}
As we can see, $p$ must have finite $L_\frac 32$-norm. The exponent $3/2$ is convenient but not a unique choice of a function class for the pressure. It would interesting to know how constants in the Wolf's condition depends on the pressure. For example,
the classical Caffarelli-Kohn-Nirenberg  condition, see \cite{CKN}, tells that there are two universal positive constants $\varepsilon_1$ and $c_1$ such  that
if 
$$\frac 1{R^2}\int\limits_{Q(z_0,R)}(|u(z)|^3+|p-[p]_{B(x_0,R)}|^\frac 32)dz<\varepsilon_1$$
then 
$$|u(z)|\leq \frac {c_1}R$$
for all $z=(x,t)\in Q(z_0,R/2)$.
The latter condition is also invariant with respect the natural Navier-Stokes scaling $u(x,t)\to\lambda u(\lambda x,\lambda^2t)$ and $p(x,t)\to\lambda^2 p(\lambda x,\lambda^2t)$. 

Now, we would like to reformulate Wolf's  condition in the above scale invariant style.
\begin{theorem}\label{mainresultinter} Let $u$ and $p$ be a suitable weak solution to the Navier-Stokes equations in $Q(z_0,R)$.
Given $M>0$, there exist positive numbers $\varepsilon_\star=\varepsilon_*(M)$ and $c_*=c_*(M)$ such that if two conditions
%\begin{equation}\label{scaledepsilonin}
$$\frac 1{R^2}\int\limits_{Q(z_0,R)}|u|^3dxdt< \varepsilon_*(M)$$ %\end{equation}
and
%\begin{equation}\label{scaledboundpressurein}
$$\frac 1{R^2}\int\limits_{Q(z_0,R)}|p-[p]_{B(x_0,R)}|^\frac 32dxdt<M$$
%\end{equation}
hold, then  $u$ is H\"older continuous is the closure of  $Q(z_0,R/2)$. Moreover,
%\begin{equation}\label{estimatein}
$$\sup\limits_{z\in Q(z_0,R/2)}|u(z)|\leq \frac {c_*(M)}R.$$
%\end{equation}
\end{theorem}
Certainly, Theorem \ref{mainresultinter} implies Wolf's condition if we let
$$\varepsilon=\varepsilon_*\Big(1+\frac 1{R^2}\int\limits_{Q(z_0,R)}|p-[p]_{B(x_0,R)}|^\frac 32dxdt\Big).$$
This type of theorems in the case of  interior regularity appeared in \cite{Ser2007}, for further developments, see, for example, \cite{Mik20010}, \cite{WZ2014}, and \cite{KRS2015}.

In the note, we shall study boundary regularity that perhaps cannot be treated by Wolf's method.

\setcounter{equation}{0}
\section{Boundary Regularity}

We shall study regularity up to a flat part of the boundary only. The following notation will be used in what follows: $B^+(x_0,R):=\{x=(x',x_3)\in \mathbb R^3:\,\,x\in B(x_0,R),\,x_{03}<x_3\}$, $B^+(r):=B^+(0,r)$, $B^+:=B^+(1)$.

\begin{definition}\label{suitableweakhalf}
A suitable weak solution $u$ and $p$ to the classical Navier-Stokes system in $Q^+(z_0,R)=B^+(x_0,R)\times ]t_0-R^2,t_0[$ possess the following properties:
%\begin{equation}\label{spaces}
$$u\in L_{2,\infty}(Q^+(z_0,R))\cap W^{1,0}_2(Q^+(z_0,R)),\qquad p\in L_\frac 32(Q^+(z_0,R));$$
%\end{equation}
%\begin{equation}\label{NSS}
$$\partial_tu+u\cdot\nabla u-\Delta u=-\nabla p,\qquad{\rm div}\,u=0$$
%\end{equation}
in $Q^+(z_0,R)$ in the sense of distributions; 
%\begin{equation}\label{boundaryconditions}
$$u(x',0,t)=0$$
%\end{equation}
for all $|x'-x'_0|<R$ and $t\in ]t_0-R^2,t_0[$, where $x'=(x_1,x_2)$;

for a.a. $t\in ]t_0-R^2,t_0[$,
$$\int\limits_{B^+(x_0,R)}\varphi^2(x,t)|u(x,t)|^2dx+2\int\limits^t_{t_0-R^2}\int\limits_{B^+(x_0,R)}\varphi^2|\nabla u|^2dxds\leq$$
%\begin{equation}\label{localenergyinequality}
$$\leq\int\limits^t_{t_0-R^2}\int\limits_{B^+(x_0,R)}|u|^2(\partial_t\varphi^2+\Delta\varphi^2)+u\cdot\nabla \varphi^2(|u|^2+2p)dxds$$
%\end{equation}
for all $\varphi \in C^\infty_0(B(x_0,R)\times ]t_0-R^2,t_0+R^2[)$.
\end{definition}

Our aim is to show the following.
\begin{theorem}\label{mainresulthalf} Let $u$ and $p$ be a suitable weak solution to the Navier-Stokes equations in $Q^+(z_0,R)$.
Given $M>0$, there exist positive numbers $\varepsilon_\star=\varepsilon_*(M)$ and $c_*=c_*(M)$ such that if two conditions
%\begin{equation}\label{scaledepsilonbound}
$$\frac 1{R^2}\int\limits_{Q^+(z_0,R)}|u|^3dxdt< \varepsilon_*(M) $$
%\end{equation}
and
%\begin{equation}\label{scaledboundpressuremain}
$$\frac 1{R^2}\int\limits_{Q^+(z_0,R)}|p-[p]_{B^+(x_0,R)}|^\frac 32dxdt<M$$
%\end{equation}
hold, then  $u$ is H\"older continuous in the closure of  $Q^+(z_0,R/2)$. Moreover,
%\begin{equation}\label{estimatehalf}
$$\sup\limits_{z\in Q^+(z_0,R/2)}|u(z)|\leq \frac {c_*(M)}R.$$
%\end{equation}
\end{theorem}

We start with the proof of the following auxiliary statement.

\begin{pro}\label{epsregcondonbound}
Let $u$ and $p$ be a suitable weak solution to the Navier-Stokes equations in $Q^+:=Q^+(0,1)$.
Given $M>0$, there exist positive numbers $\varepsilon=\varepsilon(M)$ and $c=c(M)$ such that if two conditions
%\begin{equation}\label{epsilon}
$$\int\limits_{Q^+}|u|^3dxdt< \varepsilon(M)$$
% \end{equation}
and
%\begin{equation}\label{boundpressure}
$$\int\limits_{Q^+}|p-[p]_{B^+}|^\frac 32dxdt<M$$
%\end{equation}
hold, then $z=0$ is a regular point of $u$ and therefore $u$ is H\"older continuous in the closure of a parabolic vicinity of $z=0$. Moreover,
%\begin{equation}\label{estimateinunit}
$$|u(0)|\leq c(M).$$
%\end{equation}
\end{pro}
\begin{proof}
%Assume that our statement is wrong. Then there exist $M>0$, a sequence of suitable weak solutions $u^{(m)}$ and $p^{(m)}$ in $Q^+$ such that, for any $m=1,2,...$,
%\begin{equation}\label{epsilon_m}
%\int\limits_{Q^+}|u^{(m)}|^3dxdt< \varepsilon_m=\frac 1m  \end{equation}
%and
%\begin{equation}\label{boundpressure_m}
%\int\limits_{Q^+}|p^{(m)}-[p^{(m)}]_{B^+}|^\frac 32dxdt<M
%\end{equation}
% but 
% \begin{equation}\label{blow-up}
% \limsup_{r\to0}\|u^{(m)}\|_{\infty,Q^+}>c
% \end{equation}
 %for some positive $c$ which will be chosen later.
 
 From the local energy inequality with a suitable choice of the cut-off function $\varphi$,  it follows that 
$$|u|^2_{2,Q^+(r)}:=\sup\limits_{-r^2<t<0}\int\limits_{B^+(r)}|u(x,t)|^2dx+\int\limits_{Q^+(r)}|\nabla u|^2dxdt\leq$$
 %\begin{equation}\label{energy}
$$\leq c(r)\Big[\Big(\int\limits_{Q^+}|u|^3dz\Big)^\frac 23+\int\limits_{Q^+}|u|^3dz  $$
%\end{equation}
$$+\Big(\int\limits_{Q^+}|u|^3dz\Big)^\frac 13\Big(\int\limits_{Q^+}|p-[p]_{B^+}|^\frac 32dz\Big)^\frac 23\Big]\leq  $$
 $$\leq c(r)d(\varepsilon,M)$$
 for any $r\in ]0,1[$, where $Q^+(r):=Q^+(0,r)$ and 
 $$d(\varepsilon,M):=\varepsilon^\frac 23+\varepsilon +\varepsilon^\frac 13M. $$
 Using standard multiplicative inequalities, one can show
 \begin{equation}\label{right-hand-side}
 \|u\cdot \nabla u\|_{\frac 98,\frac 32,Q^+(\frac 56)}\leq c|u|^2_{2,Q^+(\frac 56)}\leq cd(\varepsilon,M). \end{equation}

Next, as in  paper \cite{Ser2002}, let us pick up a domain $\Omega$ with smooth boundary such that $B^+(4/5)\subset\Omega\subset B^+(5/6)$ and consider the following initial boundary value problem
\begin{equation}\label{system1}
\partial_tv^1-\Delta v^1+\nabla q^1=-u\cdot \nabla u,\qquad {\rm div}\,v^1=0\end{equation}
in $Q_0=\Omega\times]-(5/6)^2,0[$ 
and 
\begin{equation}\label{bc1}
v^1=0
\end{equation}
 on $\partial' Q_0=(\Omega\times \{t=-(5/6)^2\})\cup(\partial \Omega\times [-(5/6)^2,0])$. For solutions to (\ref{system1}), (\ref{bc1}), the following estimate is valid:
\begin{equation}\label{coercive1}
\|\partial_tv^1\|_{\frac{9}{8},\frac 32,Q_0}+\|\nabla^2v^1\|_{\frac{9}{8},\frac 32,Q_0}+\|\nabla q^1\|_{\frac{9}{8},\frac 32,Q_0}\leq c\|u\cdot \nabla u\|_{\frac 98,\frac 32,Q_0}.\end{equation}
Letting $v^2=u-v^1$ and $q^2=p-q^1$, we observe that the above introduced functions satisfies the following relations
%\begin{equation}\label{system2}
$$\partial_tv^2-\Delta v^2+\nabla q^2=0,\qquad {\rm div}\,v^2=0$$
%\end{equation}
in $B^+(4/5)\times ]-(4/5)^2,0[$, and 
%\begin{equation}\label{bc2}
$$v^2(x',0,t)=0$$
%\end{equation}
for $|x'|<4/5$ and $t\in ]-(4/5)^2,0[$.
According to \cite{Ser2000}  and \cite{Ser2009}, $q^2$ obeys the estimate
$$\|\nabla q^2\|_{9,\frac 32,Q^+(3/4)}\leq 
c(\|\nabla v^2\|_{\frac 32, Q^+(4/5)}+\| v^2\|_{\frac 32, Q^+(4/5)}$$
%\begin{equation}\label{est2}
$$+\| q^2-[q^2]_{B^+(4/5)}\|_{\frac 32, Q^+(4/5)})\leq c(\|\nabla u\|_{\frac 32, Q^+(4/5)}+\| u\|_{\frac 32, Q^+(4/5)}+$$
%\end{equation}
$$+\|p-[p]_{B^+(4/5)}\|_{\frac 32, Q^+(4/5)} +\|\nabla v^1\|_{\frac 32, Q^+(4/5)}+\| v^1\|_{\frac 32, Q^+(4/5)}+$$
$$+\| q^1-[q^1]_{B^+(4/5)}\|_{\frac 32, Q^+(4/5)}).$$
Assuming $0<\varepsilon<1$, we find elementary bounds:
$$\|\nabla u\|_{\frac 32, Q^+(4/5)}+\| u\|_{\frac 32, Q^+(4/5)}\leq c|u|_{2,Q^+}\leq c(M),$$
$$\|p-[p]_{B^+(4/5)}\|_{\frac 32, Q^+(4/5)}\leq c\|p-[p]_{B^+}\|_{\frac 32, Q^+}\leq cM.$$
Next,  from (\ref{right-hand-side}),  (\ref{coercive1}), %the parabolic Poincare inequality,
 and the elliptic embedding, it follows that:
$$\|\nabla v^1\|_{\frac 32, Q^+(4/5)}+\| v^1\|_{\frac 32, Q^+(4/5)}\leq c%\|\partial_tv^1\|_{\frac{9}{8},\frac 32,Q_0}+$$$$+
\|\nabla^2v^1\|_{\frac{9}{8},\frac 32,Q_0}\leq c(M)$$
and 
$$\| q^1-[q^1]_{B^+(4/5)}\|_{\frac 32, Q^+(4/5)})\leq \|\nabla q^1\|_{\frac 98,\frac 32,Q^+(4/5)}\leq c(M).$$
So, finally, we find
\begin{equation}\label{pressure2}
\|\nabla q^2\|_{9,\frac 32,Q^+(3/4)}\leq c(M).\end{equation}
It is worthy to notice  that  the right hand side is independent of $\varepsilon $.

We then  have for $0<r<2/3$
$$I_\varepsilon(r):=\frac 1{r^2}\int\limits_{Q^+(r)}(|u|^3+|p-[p]_{B^+(r)}|^\frac 32)dz\leq $$
$$\leq c\frac 1{r^2}\Big(\varepsilon +
\int\limits_{Q^+(r)}(|q^1-[q^1]_{B^+(r)}|^\frac 32+|q^2-[q^2]_{B^+(r)}|^\frac 32)dz\Big).$$
By Poincare inequality,
$$\frac 1{r^2}
\int\limits_{Q^+(r)}|q^1-[q^1]_{B^+(r)}|^\frac 32dz\leq c\frac 1{r^\frac 32}
\int\limits^0_{-r^2}\Big(\int\limits_{B^+(r)}|\nabla q^1|^\frac 98dx\Big)^\frac 43dt\leq $$$$\leq c\frac 1{r^\frac 32}\|\nabla q^1\|^\frac 32_{\frac 98,\frac 32, Q_0}\leq  c\frac 1{r^\frac 32}d(\varepsilon,M) .$$
For the second part of the pressure, we have the same arguments 
plus estimate (\ref{pressure2})
$$\frac 1{r^2}
\int\limits_{Q^+(r)}|q^2-[q^2]_{B^+(r)}|^\frac 32dz\leq cr^2\int\limits^0_{-r^2}\Big(\int\limits_{B^+(r)}|\nabla q^2|^9dx\Big)^\frac 16dt\leq $$$$\leq cr^2\|\nabla q^2\|_{9,\frac 32,Q^+(3/4)}\leq c(M)r^2.$$
So, we find
\begin{equation}\label{criterii}
I_\varepsilon(r)\leq c\Big(\frac 1{r^2}\varepsilon+\frac 1{r^\frac 32}d(\varepsilon,M)\Big)+c(M)r^2
\end{equation}
for any $0<\varepsilon<1$ and for any $0<r<2/3$.

Let us state the following condition of local boundary regularity proved in \cite{Ser2002}.
\begin{pro}\label{maincondition}
Let $w$ and $\pi$ be a suitable weak solution to the Navier-Stokes system in 
$Q^+(R)$. There exist two universal positive constants $\varepsilon_0$ and $c_0$ such that
if 
$$\frac 1{R^2}\int\limits_{Q^+(R)}(|w|^3+|\pi-[\pi]_{B^+(R)}|^\frac 32)dz<\varepsilon_0,$$
then  the function $z\mapsto w(z)$ is H\"older continuous in the closure of $Q^+(R/2)$ and 
$$\sup\limits_{z\in Q^+(R/2)}|w(z)|\leq \frac {c_0}R.$$
\end{pro}
Let us select a positive number $r=r(M)<1/2$ such that
$$c(M)r^2<\frac {\varepsilon_0}2.$$
Then we can pick up $\varepsilon=\varepsilon(M)  $ such that
$$c\Big(\frac 1{r(M)^2}\varepsilon+\frac 1{r(M)^\frac 32}d(\varepsilon,M)\Big)<\frac {\varepsilon_0}2.$$
From Proposition \ref{maincondition} and from (\ref{criterii}), it follows that 
the  function $z\mapsto u(z)$ is H\"older continuous in the closure of $Q^+(r(M)/2)$ and 
$$|u(0)|\leq c(M)=2c_0/r(M).$$
\end{proof}

The scaled version of Proposition \ref{epsregcondonbound} is as follows.
\begin{pro}\label{scaledepsregcondonbound} Let $u$ and $p$ be a suitable weak solution to the Navier-Stokes equations in $Q^+(R)$.
Given $M>0$, there exist positive numbers $\varepsilon=\varepsilon(M)$ and $c=c(M)$ such that if two conditions
%\begin{equation}\label{scaledepsilonorigin}
$$\frac 1{R^2}\int\limits_{Q^+(R)}|u|^3dxdt< \varepsilon(M) $$
%\end{equation}
and
%\begin{equation}\label{scaledboundpressurespecialcase}
$$\frac 1{R^2}\int\limits_{Q^+(R)}|p-[p]_{B^+(R)}|^\frac 32dxdt<M$$
%\end{equation}
hold, then $z=0$ is a regular point of $u$ and therefore $u$ is H\"older continuous in the closure of a parabolic vicinity of $z=0$. Moreover,
%\begin{equation}\label{estimatescaled}
$$|u(0)|\leq \frac {c(M)}R.$$
%\end{equation}
\end{pro}
Now, we wish to show the following.
\begin{pro}\label{epsregcondonbound1}
Let $u$ and $p$ be a suitable weak solution to the Navier-Stokes equations in $Q^+$. Given
 $M>0$, there exist positive numbers $\varepsilon_1=\varepsilon_1(M)$ and $c_1=c_1(M)$ such that if two conditions
%\begin{equation}\label{epsilon1}
$$\int\limits_{Q^+}|u|^3dxdt< \varepsilon_1(M) $$
%\end{equation}
and
%\begin{equation}\label{boundpressure1}
$$\int\limits_{Q^+}|p-[p]_{B^+}|^\frac 32dxdt<M$$
%\end{equation}
hold, then $u$ is H\"older continuous in the closure of $Q^+(1/2)$. Moreover,
%\begin{equation}\label{estimatefinal}
$$\sup\limits_{Q^+(1/2)}|u(z)|\leq c_1(M).$$
%\end{equation}
\end{pro}
\begin{proof}
For $z_0=(x_0,t_0)\in \overline{Q}^+(1/2)$, we have 
$$\frac 1{(1/2)^2}\int\limits_{Q^+(z_0, 1/2)}|u|^3dz\leq 4\int\limits_{Q^+}|u|^3dz<4\varepsilon_1(M) $$
and
$$\frac 1{(1/2)^2}\int\limits_{Q^+(z_0, 1/2)}|p-[p]_{B^+(x_0)}|^\frac 32dz\leq 4c\int\limits_{Q^+}|p-[p]_{B^+}|^\frac 32dz\leq 4cM.$$
We complete the proof by letting 
$$\varepsilon_1(M)=\frac 1r\varepsilon(4cM), \qquad c_1(M)=2c(4cM).$$
\end{proof}
Now,   Theorem \ref{mainresulthalf} follows from obvious scaling and shift and from Proposition 
\ref{epsregcondonbound1}.

\end{document}